\documentclass[12pt,reqno]{amsart}
\usepackage{graphicx}
\usepackage{amssymb,amscd, amsthm, latexsym, mathrsfs, epsfig}
\usepackage[all]{xy}
\usepackage{todonotes}

\newcommand\PP{\mathbb P}

\newcommand\C{\mathbb C}

\newcommand\Q{\mathbb Q}
\newcommand\R{\mathbb R}
\newcommand\Z{\mathbb Z}

\newcommand{\M}{\mathcal{M}}

\newcommand{\eps}{\epsilon}

\newcommand\Aut{\operatorname{Aut}}
\newcommand\ad{\operatorname{ad}}
\newcommand\dt{\operatorname{DT}}

\newcommand\bra{\langle}
\newcommand\ket{\rangle}
\newcommand{\pow}[1]{[\![ {#1} ]\!]}

\newcommand{\del}{\partial}

\newcommand{\Hom}{\operatorname{Hom}}

\newcommand{\Coh}{\operatorname{Coh}}

\makeatletter \@addtoreset{equation}{section} \makeatother

\newtheorem{thm}{Theorem}[section]

\newtheorem{lem}[thm]{Lemma}

\theoremstyle{definition}
\newtheorem{definition}[thm]{Definition}

\newtheorem{rmk}[thm]{Remark}

\newcommand\narrowdots{\hbox to 1em{.\hss.\hss.}}

\title[ ]{A note on BPS structures and Gopakumar-Vafa invariants}

\author{Jacopo Stoppa}

\email{jstoppa@sissa.it}
\address{SISSA, via Bonomea 265, 34136 Trieste, Italy} 
\address{Institute for Geometry and Physics (iGAP), Via Beirut 2-4, 34014 Trieste, Italy}
\begin{document}

\maketitle

{\centering\footnotesize \emph{Dedicated to the memory of Boris Dubrovin} \par}

\begin{abstract} We regard the work of Maulik and Toda, proposing a sheaf-theoretic approach to Gopa\-kumar-Vafa invariants, as defining a BPS structure, that is, a collection of BPS invariants together with a central charge. Assuming their conjectures, we show that a canonical flat section of the flat connection corresponding to this BPS structure, at the level of formal power series, reproduces the Gromov-Witten partition function for all genera, up to some error terms in genus $0$ and $1$. This generalises a result of Bridgeland and Iwaki for the contribution from genus $0$ Gopakumar-Vafa invariants.
\end{abstract}

\section{Introduction}

The abstract notions of a \emph{BPS structure} and of its \emph{variation} package some important properties of enumerative invariants of Donaldson-Thomas type for Calabi-Yau threefolds \cite{js, ks}. These notions are a special case of the more general \emph{stability data} introduced by Kontsevich and Soibelman \cite{ks}; the terminology is due to Bridgeland \cite{bridRH}.

A given BPS structure defines \emph{BPS automorphisms} of a certain infinite-dimensional algebra, and these should be regarded as defining in a very natural way the monodromy of a flat connection $\nabla^{\rm{BPS}}$ on an infinite-dimensional principal bundle over $\PP^1$. This basic intuition was developed in several works including \cite{Anna, AnnaJ, bridRH, bt_stab, fgs, ks}. 

In the present note we are only concerned with a very special case of this theory, so we can be almost self-contained: the relevant background material may be found in Section \ref{BackSec}. It is important to point out that we will work at the level of formal power series, in a sense that will be made clear in the following (and as explained, for example, in \cite{AnnaJ, fgs}, and in \cite{Anna} Section 4.2). This applies in particular to the connection $\nabla^{\rm{BPS}}$ and to its flat sections. We do not try to prove ``non-perturbative" results in the spirit of \cite{bridRH2}: this is a much harder, open problem, in general.

The purpose of this note is to observe that the recent work of Maulik and Toda \cite{mautoda, todaGV}, concerning a geometric, sheaf-theoretic, definition of Gopakumar-Vafa invariants, leads naturally to the introduction of a corresponding BPS structure. The work of Maulik and Toda is briefly reviewed in Section \ref{MauTodaSec}. The principal virtue of this BPS structure is explained by our main result, Theorem \ref{MainThm}, which we summarise here in somewhat imprecise terms. 
\begin{thm}[see Theorem \ref{MainThm}]\label{MainThmIntro} Let us work at the level of formal power series, and assume the conjectures of Maulik and Toda. Then the canonical flat section, along the ray $\R_{>0}$, of the flat connection $\nabla^{\rm{BPS}}$ attached to the Maulik-Toda BPS structure, has an expansion around $t = 0$ which reproduces the Gromov-Witten partition function expressed in Gopakumar-Vafa form, except for some error terms due to Gopakumar-Vafa contributions to Gromov-Witten invariants of genus $0$ and $1$.
\end{thm}
Theorem \ref{MainThm} is proved in Section \ref{MainSec}. As we will explain it can be seen as a generalisation of a computation due to Bridgeland and Iwaki (see \cite{bridRH} Section 6.3) for the contribution from genus $0$ Gopakumar-Vafa invariants. Again that computation misses the Gopakumar-Vafa contributions to Gromov-Witten invariants of genus $0$ and $1$.
\begin{rmk} A more classical approach to Gopakumar-Vafa invariants uses Pandharipande-Thomas stable pairs \cite{ptBPS}. An analogue of Theorem \ref{MainThm} for stable pairs is studied in \cite{jPairs}. It is interesting to compare the two approaches. The Gopakumar-Vafa contributions emerge in two very different ways. The genus $0$ and $1$ cases also play a special role in the stable pairs approach: for example the genus $0$ case leads to divergent integrals that require appropriate regularisation.
\end{rmk}
The proof of Theorem \ref{MainThm} will show that (a specialisation of) the canonical flat section of $\nabla^{\rm BPS}$ along $\R_{>0}$ may be written as a sum of contributions from effective curve classes $\beta$, and that each such contribution can be regarded as a well-defined function of $t \in \R_{>0}$, with values in a certain ring of formal power series. In particular it makes sense to consider the large $t$ asymptotics of each fixed $\beta$ contribution. This is reminiscent of computations of the entries of the ``central connection matrix" for the Dubrovin connection in quantum cohomology \cite{cotti, dubrovin, gamma}, of which $\nabla^{\rm BPS}$ is an analogue. The leading order term for these asymptotics is worked out in Section \ref{Large_t_Sec}. We point out that it defines in a natural way a formal family of automorphisms of an algebraic torus, of the type studied in \cite{gps}.\\
 
\noindent{\textbf{Acknowledgements.}} I would like to thank A. Barbieri, T. Bridgeland, M. Kool and J. Scalise for some discussions related to this note, and the anonymous Referees for several important comments on the manuscript. I am very grateful to the organisers and participants in the workshops \emph{Geometric correspondences of gauge theories}, Vienna, August-September 2018, and \emph{D-modules, quantum geometry, and related topics}, Kyoto, December 2018, especially to H. Iritani, K. Iwaki, T. Mochizuki, C. Sabbah and Y. Toda, for their interest in my talks and for their suggestions. My participation in these activities was supported by the Erwin Schr\"odinger Institute, Vienna and the Research Institute for the Mathematical Sciences, Kyoto.\\ 

\section{Basic notions}\label{BackSec}

\begin{definition} A \emph{BPS structure} $(\Gamma, \bra - , -\ket, Z, \Omega)$ is the datum of a finite rank lattice $\Gamma$ (the \emph{charge lattice}), endowed with a skew-symmetric bilinear form $\bra-,-\ket$ with values in $\Z$ (the \emph{intersection form}), together with a group homomorphism $Z\!:\Gamma \to \C$ (the \emph{central charge}) and a map of sets $\Omega\!: \Gamma \to \Q$ (the \emph{BPS spectrum}). 
\end{definition}
Usually the additional \emph{support condition} is imposed: fixing some arbitrary norm $|| - ||$ on $\Gamma \otimes \R$, there is a uniform constant $C > 0$ such that for all $\alpha$ with $\Omega(\alpha) \neq 0$, 
\begin{equation*}
|Z(\alpha)| > C || \alpha ||.
\end{equation*} 
In fact this does not play a role in the present note.
\begin{definition} A BPS structure is \emph{integral} if $\Omega$ takes values in $\Z$, and \emph{symmetric} if we have $\Omega(\alpha) = \Omega(-\alpha)$ for all $\alpha \in \Gamma$.
\end{definition}

\begin{definition} A BPS structure is \emph{uncoupled} if the locus in $\Gamma$ where $\Omega$ does not vanish is isotropic: that is for all $\alpha, \beta \in \Gamma$ with $\Omega(\alpha), \Omega(\beta) \neq 0$ one has $\bra \alpha, \beta \ket = 0$.
\end{definition}

\begin{definition} A \emph{framed variation of uncoupled BPS structures} over a complex or real analytic manifold $M$ is a family of uncoupled BPS structures of the form $(\Gamma_p, \bra - , - \ket_p, Z_p, \Omega_p)$, parametrised by $p \in M$, such that $\Gamma_p$, $\bra - , - \ket_p$ and $\Omega_p$ are all \emph{constant} in $p$, while $Z_p \in \Hom(\Gamma, \C)$ varies \emph{holomorphically} or \emph{real analytically} with $p$. 
\end{definition}
\begin{definition} We denote by $\C[\Gamma]$ the group-algebra of $\Gamma$ endowed with the usual commutative product, twisted by the form $\bra - , -\ket$,
\begin{equation*}
x_{\alpha} x_{\beta} = (-1)^{\bra \alpha, \beta\ket} x_{\alpha + \beta}
\end{equation*}
and with the Poisson bracket
\begin{equation*}
[x_{\alpha}, x_{\beta}] = (-1)^{\bra \alpha, \beta\ket} \bra \alpha,\beta\ket x_{\alpha + \beta}.
\end{equation*}
We let $\C[\Gamma]\pow{s}$ denote the ring of formal power series in one variable over $\C[\Gamma]$, and extend the above commutative product and Poisson bracket to $\C[\Gamma]\pow{s}$ by $\C\pow{s}$-linearity.
\end{definition}
Fix a norm $|| - ||$ on $\Gamma \otimes\R$ which takes integer values on $\Gamma \subset \Gamma \otimes\R$.
\begin{definition} For a fixed uncoupled BPS structure, the \emph{BPS automorphism} $\mathbb{S}_{\ell} \in \Aut_{\C\pow{s}}(\C[\Gamma]\pow{s})$ attached to a ray $\ell \subset \C^*$ is defined by
\begin{equation*}
\mathbb{S}_{\ell}(x_{\alpha}) = x_{\alpha}\prod_{Z(\beta) \in \ell}(1- s^{|| \beta ||} x_{\beta})^{\Omega(\beta)\bra\alpha,\beta\ket}.
\end{equation*}
\end{definition}
\begin{rmk} Here we make rigorous sense of the automorphisms $\mathbb{S}_{\ell}$, and later on, of the BPS connection $\nabla^{\rm BPS} $ and its flat sections, by working with formal power series in $s$ and using the arbitrary norm $|| - ||$. This approach is \emph{not} compatible with variations, except in the uncoupled case. We refer to \cite{AnnaJ, bridRH, ks} for more details on the general case. 
\end{rmk}
In what follows we denote by $\operatorname{Der}_{\C\pow{s}}(\C[\Gamma]\pow{s})$ the module of derivations as a commutative algebra. Note that a given central charge $Z$ can be regarded as such a derivation acting by $Z(x_{\alpha}) = Z(\alpha) x_{\alpha}$. 

It turns out that $\mathbb{S}_{\ell}$ is a Poisson automorphism, and in fact it can be expressed uniquely in the form
\begin{equation*}
\mathbb{S}_{\ell} = \exp_{\operatorname{Der}_{\C\pow{s}}(\C[\Gamma]\pow{s})}\left(\ad \sum_{Z(\alpha) \in \ell} \dt(\alpha) s^{||\alpha||}x_{\alpha} \right)
\end{equation*}
for certain ``Donaldson-Thomas" rational numbers $\dt(\alpha) \in \Q$. 
\begin{definition} The \emph{BPS flat connection} attached to a fixed uncoupled BPS structure is the meromorphic connection $\nabla^{\rm BPS}$ on the trivial principal $\Aut_{\C\pow{s}}(\C[\Gamma]\pow{s})$-bundle over $\PP^1$, of the form
\begin{equation*}
\nabla^{\rm BPS} = d - \left(\frac{Z}{t^2} + \frac{f}{t}\right) dt
\end{equation*} 
for some $f \in \operatorname{Der}_{\C\pow{s}}(\C[\Gamma]\pow{s})$, and such that its generalised monodromy at $t = 0$ is given by the collection of rays and automorphisms $\{ \ell, \mathbb{S}_{\ell}\}$. 
\end{definition}
One can prove that for a fixed $Z$ such a connection exists and is unique (see \cite{fgs}). 
\begin{definition} Fix a ray $r \subset \C^*$, distinct from all $\ell \subset \C^*$ for which $\mathbb{S}_{\ell} \neq \operatorname{Id}$. The \emph{canononical flat section} of $\nabla^{\rm BPS}$ along $r$ is a real analytic function $X_r(t)$, defined for $t \in r$, with values in $\Aut_{\C\pow{s}}(\C[\Gamma]\pow{s})$, such that $\nabla^{\rm BPS} X_r(t) = 0$ and we have $\exp(-Z/t) X_r(t) \to \operatorname{Id}$ as $t\to 0$. 
\end{definition} 
Again one can show that this exists and is unique for a fixed $Z$ (see \cite{fgs}), but in fact in the present uncoupled case it is quite easy to write down the flat section explicitly. Namely, for each element $\alpha \in \Gamma$ we introduce a function $\operatorname{F}_{\alpha}(t)$, given by
\begin{align*}
\operatorname{F}_{\alpha}(t) = \Omega(\alpha)\frac{t}{2\pi i}  \int_{\R_{>0}Z(\alpha)} \frac{\log(1 -  s^{||\alpha||} x_{\alpha} e^{- Z(\alpha)/z})}{z-t}\frac{dz}{z}.
\end{align*}
Here the logarithm appearing in the integrand must be interpreted as  the usual formal power series expansion in $s$. It is straightforward to check that $\operatorname{F}_{\alpha}(t)$ is a well-defined holomorphic function of $t \in \C^*\setminus \{\pm \R_{>0}Z(\alpha)\}$, with values in $\C[\Gamma]\pow{s}$. Let $r \subset \C^*$ be a ray distinct from $\R_{<0}$ and all the $\ell \subset \C^*$ for which $\mathbb{S}_{\ell} \neq \operatorname{Id}$. For all $t \in r$, we define a $\C\pow{s}$-linear endomorphism $X_{r}(t)$ of $\C[\Gamma]\pow{s}$ by
\begin{equation}\label{GeneralRHSol}
X_{r}(t)(x_{\alpha}) = x_{\alpha }e^{Z(\alpha)/t} \exp\left(\sum_{\beta \in \Gamma} \bra \alpha, \beta \ket \operatorname{F}_{\beta}(t) \right).
\end{equation}
One can show that, for all $\alpha$, $X_{r}(t)(x_{\alpha})$ is a real-analytic function of $t \in r$, with values in $\C[\Gamma]\pow{s}$, and that $X_{r}(t)$ defines a canonical flat section of $\nabla^{\rm BPS}$ \cite{AnnaJ, fgs}. The crucial point for the latter claim is that, working modulo any power of the formal parameter $s$, $X_{r}(t)$ extends to a homolorphic function of $t$ in the complement of finitely many rays $\ell \subset \C^*$, with branch-cut discontinuities along $\ell$ given by the relation between clock-wise and counter-clockwise limits 
\begin{equation*}
X^+_{r}(t) = \mathbb{S}_{\ell}\circ X^-_{r}(t). 
\end{equation*}
This follows from a version of the Cauchy formula (sometimes called Plemelj's formula).   

\section{Maulik-Toda variation of BPS structure}\label{MauTodaSec} 

Let $X$ be a Calabi-Yau threefold. Maulik and Toda \cite{mautoda} proposed a new approach to the problem of providing a sheaf-theoretic definition of the Gopakumar-Vafa invariants of $X$. 

Firstly they observed that in all existing sheaf-theoretic approaches (such as those due to Hosono-Saito-Takahashi \cite{hosono} and Kiem-Li \cite{kiem}), once sheaf-theoretic invariants $\mathcal{S}_n(\beta)$ have been introduced for $n \in \Z$, $\beta \in H_2(X, \Z)$, the corresponding conjectural relation to the Gopakumar-Vafa invariants $n_{g, \beta}$ is always captured by the relations
\begin{equation}\label{SheafSymm}
\mathcal{S}_{n}(\beta) = \mathcal{S}_{-n}(\beta)
\end{equation}
and
\begin{equation}\label{SheafGV}
\sum_{n \in \Z} \mathcal{S}_{n}(\beta) y^n = \sum_{g\geq0} n_{g, \beta} (y^{1/2} + y^{-1/2})^{2g}.  
\end{equation}
Note that the first property \eqref{SheafSymm} implies that one can always find unique integers $n_{g, \beta}$ so that the second relation \eqref{SheafGV} is satisfied.

Secondly, fixing a K\"ahler form $\omega$, Maulik and Toda proposed to refine the existing sheaf-theoretic approaches by considering the fine moduli space $M_{1}(\beta)$, parametrising Gieseker $\omega$-semistable torsion sheaves $F \in \Coh(X)$ with dimension $1$, fundamental class $[F] = \beta$ and holomorphic Euler characteristic $\chi(F) = 1$, together with its Hilbert to Chow morphism
\begin{equation*}
\pi\!: M_{1}(\beta) \to \rm{Chow}(\beta).
\end{equation*} 
Denoting by $\Phi_{M_{1}(\beta)}$ the perverse sheaf of vanishing cycles on $M_{1}(\beta)$ (constructed in \cite{brav, kiem}), the proposed sheaf-theoretic invariants are given by
\begin{equation}
\mathcal{S}_{n}(\beta) = \chi(^{p}\mathcal{H}^n (\textbf{R}\pi_*\Phi_{M_{1}(\beta)})).
\end{equation} 
Property \eqref{SheafSymm} then follows from the self-duality of $\Phi_{M_{1}(\beta)}$ and Verdier duality, and the sums appearing in \eqref{SheafGV} are automatically \emph{finite}.
\begin{rmk} This definition of $\mathcal{S}_{n}(\beta)$ is a refinement of the torsion-sheaf invariants enumerating curves in class $\beta$ constructed by Joyce and Song (see \cite{js} Section 6.4), which are given by $\chi(\Phi_{M_{1}(\beta)})$. Conjecturally one has an equivalence with Gopakumar-Vafa invariants $\chi(\Phi_{M_{1}(\beta)}) = n_{0, \beta}$ (loc. cit. Conjecture 6.20). Note that this is compatible with specialising \eqref{SheafGV} to $y = -1$.
\end{rmk}
Importantly for us Maulik and Toda \cite{mautoda, todaGV} conjecture that this new sheaf-theoretic definition is very robust. In particular it should be independent of the parameters in the construction in two crucial ways:
\begin{enumerate}
\item[\textbf{C$\omega$}:] the invariants $\mathcal{S}_{n}(\beta)$ should not depend on the choice of K\"ahler form $\omega$ used in defining Gieseker semistability;
\item[\textbf{C$\chi$}:] the fine moduli space $M_{1}(\beta)$ can be replaced by any coarse moduli space $M_{m}(\beta)$ of Gieseker $\omega$-semistable sheaves $F\in \Coh(X)$ with $[F] = \beta$, $\chi(F) = m \in \Z$, without affecting the invariants $\mathcal{S}_{n}(\beta)$, provided the perverse sheaf of vanishing cycles $\Phi_{M_{1}(\beta)}$ is replaced by a suitable perverse sheaf $\Phi_{M_{m}(\beta)}$, descended from the moduli stack $\M_{m}(\beta)$. 
\end{enumerate}

These conjectural properties provide strong motivation for introducing a variation of BPS structures corresponding to the Maulik-Toda construction.  
\begin{definition}\label{MauTodaBPSDef} The Maulik-Toda BPS structure $(\Gamma, \bra - , - \ket, Z, \Omega)$, for a fixed K\"ahler form $\omega$ and complex parameter $\eps$, is given by the following choices:
\begin{enumerate}
\item[$\bullet$] the charge lattice is $\Gamma = \bigoplus^3_{i = 0} H_{2i}(X, \Z) \oplus E$, where the extra summand $E$ is a copy of $\Z$;
\item[$\bullet$] the intersection form $\bra - , -\ket$ is the Euler form on $\bigoplus^3_{i = 0} H_{2i}(X, \Z)$, extended by zero to $\Gamma$;
\item[$\bullet$] on $H_0(X, \Z) \oplus H_2(X, \Z)\oplus E \subset \bigoplus^3_{i = 0} H_{2i}(X, \Z) \oplus E$ the central charge is given by  
\begin{equation*}
Z(m, \beta, n) = Z((m, \beta,0,0), n) = \int_{\beta} (i\omega) - m +  n\eps, 
\end{equation*}
and this is extended by zero to the rest of $\Gamma$;
\item[$\bullet$] on $H_0(X, \Z) \oplus NE(X) \oplus E \subset \bigoplus^3_{i = 0} H_{2i}(X, \Z) \oplus E$ the BPS spectrum is given by
\begin{align*}
\Omega(m, \beta, n) &= \Omega((m, \beta, 0, 0), n)\\& = (-1)^n\chi(^{p}\mathcal{H}^n (\textbf{R}\pi_*\Phi_{M_{m}(\beta)})).
\end{align*}
The spectrum is then extended by the symmetry property
\begin{equation*}
\Omega(m, -\beta, n) = \Omega((m, -\beta, 0, 0), n) = \Omega(-m, \beta, - n), 
\end{equation*}
for all $\beta \in NE(X)$, and further extended by zero to the rest of the charge lattice $\Gamma$.
\end{enumerate}
\end{definition}
\begin{rmk} Note that the BPS invariants underlying our structure vanish if $\pm\beta$ is not an effective curve class. In particular we chose to set $\Omega(m, 0, n) = 0$ for all $m, n \in \Z$. A more natural choice would be to let instead $\Omega(m, 0, n) = -\chi(X)$, the invariant virtually enumerating zero-dimensional torsion sheaves (see e.g. \cite{js}, Section 6.3). But this choice is irrelevant for our main result Theorem \ref{MainThm}, which is concerned with the case of an effective curve class, and would have only complicated Definition \ref{MauTodaBPSDef}. The case $\beta = 0$ is already covered by the computation of Bridgeland and Iwaki in \cite{bridRH}, Section 6.3. 
\end{rmk}
\begin{lem} Assume the condition  \emph{\textbf{C$\omega$}}. Then the family of Maulik-Toda BPS structures parametrised by a K\"ahler form $\omega$ and a complex parameter $\eps$ is a framed variation of integral, uncoupled BPS structures.
\end{lem}
\begin{proof} Consider a BPS structure in the family. It is clearly integral, by definition. Moreover the only elements of $\Gamma$ for which $\Omega$ does not vanish lie in $H_0(X, \Z) \oplus H_2(X, \Z)\oplus E \subset \bigoplus^3_{i = 0} H_{2i}(X, \Z) \oplus E$. The intersection form is trivial along $E$, by definition, and restricts to the Euler pairing on $H_0(X, \Z) \oplus H_2(X, \Z)$. But the latter pairing vanishes because $X$ is a threefold. It follows that each BPS structure in the family is uncoupled. Finally the central charge $Z$ varies real analytically with $\omega$, $\eps$, by definition, while under the condition \textbf{C$\omega$} the BPS spectrum $\Omega$ is constant.
\end{proof}
\begin{definition} Whenever we assume the condition \textbf{C$\chi$} we will write $\Omega_n(\beta)$ for the common value of all the invariants $\Omega(m, \beta, n)$, where $\beta \in \pm NE(X)$. 
\end{definition}
\begin{lem}\label{nSymmLem} Assume the condition \emph{\textbf{C$\chi$}}. Then we have
\begin{equation*}
\Omega_n(\beta) = \Omega_{-n}(\beta).
\end{equation*}
\end{lem} 
\begin{proof} Under the condition \textbf{C$\chi$} we have 
\begin{equation*}
\Omega_n(\beta) = (-1)^n\chi(^{p}\mathcal{H}^n (\textbf{R}\pi_*\Phi_{M_{1}(\beta)})).
\end{equation*}
The required symmetry with respect to $n$ then follows from the self-duality of the object $\Phi_{M_{1}(\beta)}$ and Verdier duality, as observed by Maulik and Toda \cite{mautoda}.
\end{proof}

\section{Main result}\label{MainSec} 
In this Section we introduce and prove our main result, a more precise version of Theorem \ref{MainThmIntro} in the Introduction.

We consider a Maulik-Toda BPS structure in the sense of Definition \ref{MauTodaBPSDef}, for a fixed K\"ahler form $\omega$ and \emph{real} parameter $\eps \in \R$. 

Let $\nabla^{\rm{BPS}}$ denote the corresponding BPS flat cannection. As discussed in Section \ref{BackSec} this is a well-defined flat connection on the trivial principal bundle over $\PP^1$, with structure group $\Aut(\C[\Gamma]\pow{s})$, and with singularities at $t = 0$, $t = \infty$ (denoting by $t$ a coordinate on $\C^* \subset \PP^1$). Note that as we are assuming $\eps \in \R$ the real line does not contain central charges $Z(m, \beta, n)$. A canonical flat section along the positive real line $\R_{>0}$ is a real analytic function $X(t)$, with values in $\Aut(\C[\Gamma]\pow{s})$, which can be described explicitly. 
\begin{definition}\label{FrequDef} For each element 
\begin{equation*}
(m, \beta, n) \in H_0(X, \Z) \oplus NE(X) \oplus E
\end{equation*}
we introduce two functions $\operatorname{F}^{\pm}_{m, \beta, n}(t,Z)$, given by
\begin{align*}
&\operatorname{F}^{\pm}_{m, \beta, n}(t,Z) =\\&= \pm \Omega(m,\beta,n)\frac{t}{2\pi i}  \int_{\pm\R_{>0}Z(m,\beta,n)} \frac{\log(1 -  s^{||(m, \beta, n)||} x_{\pm(m, \beta, n)} e^{\mp Z(m,\beta,n)/z})}{z-t}\frac{dz}{z}.
\end{align*}
It is straightforward to check that $\operatorname{F}^{\pm}_{m, \beta, n}(t)$ are well-defined functions of $t \in \R_{>0}$ with values in $\C[\Gamma]\pow{s}$.
\end{definition}
Let $[D] \in H_4(X, \Z)$ denote a divisor class, with a corresponding monomial $x_{[D]} \in \C[\Gamma]$. 
\begin{lem} A canonical flat section $X(t)$ of $\nabla^{\rm BPS}$ along the positive real line is determined by the identity
\begin{equation*}
 X(t)( x_{[D]}) = x_{[D]}\exp\left(\sum_{\beta \in NE(X)} \bra [D], \beta \ket \left(\sum_{m, n\in\Z}\operatorname{F}^+_{m, \beta, n}(t,Z) + \operatorname{F}^-_{m, \beta, n}(t,Z)\right)\right).
\end{equation*}
\end{lem}
\begin{proof} This is a special case of the general solution \eqref{GeneralRHSol}.
\end{proof}
Our main result concerns a specialisation of this explicit formula for canonical solutions of $\nabla^{\rm BPS}$. It is convenient to introduce the function
\begin{equation*}
\left[2\sin(x/2))^{-2}\right]_{+} = (2\sin(x/2))^{-2} - \frac{1}{x^2} + \frac{1}{6}, 
\end{equation*}
as well as the notation
\begin{equation*}
v_{\beta} = i\int_{\beta} \omega. 
\end{equation*}
\begin{thm}\label{MainThm} Assume the condition \emph{\textbf{C$\chi$}}. For each $\beta \in NE(X)$, consider the specialisation given by
\begin{align*}
&\operatorname{F}_{\beta}(t,Z)=\\& = \sum_{m, n\in\Z} \big(\operatorname{F}^+_{m, \beta, n}(t, Z) + \operatorname{F}^-_{m, \beta, n}(t, Z)\big) | \big( s = x_{(m, \beta, n)} = 1\big).
\end{align*} 
Then, $\operatorname{F}_{\beta}(t, Z)$ is a well-defined function of $t, Z$, and it satisfies the differential equation
\begin{align*}
&\frac{i}{2\pi}\frac{\del}{\del t}\operatorname{F}_{\beta}(t, Z)|\big(\eps = \frac{u}{2\pi}, t = \frac{u}{ (2\pi)^2}\big) = \\&= \frac{\del}{\del v_{\beta}}\sum_{g\geq0}n_{g,\beta} \sum_{r >0}\frac{1}{r}(2 \sin(ru/2))^{2g} \left[2\sin(r u/2))^{-2}\right]_{+} e^{2\pi i r  v_{\beta} }
\end{align*}
in the sense of formal power series in the variables $u$, $e^{2\pi i v_{\beta}}$.
\end{thm}
\begin{rmk} The formal power series 
\begin{equation*}
\sum_{g\geq0}n_{g,\beta} \sum_{r >0}\frac{1}{r}(2 \sin(ru/2))^{2g-2} e^{2\pi i r  v_{\beta} }
\end{equation*}
is precisely the contribution of $\beta$ to the Gromov-Witten partition function written in Gopaku\-mar-Vafa form. Theorem \ref{MainThm} says that this can be easily computed from the asymptotics of a flat section of $\nabla^{\rm BPS}$, through the formal power series $\operatorname{F}_{\beta}(u)$, up to the missing term
\begin{equation*}
\sum_{g\geq0}n_{g,\beta} \sum_{r >0}\frac{1}{r}(2 \sin(ru/2))^{2g} \left(\frac{1}{r^2u^2 }- \frac{1}{6}\right) e^{2\pi i r  v_{\beta} }.
\end{equation*}
Note that the term $1/(r^2 u^2)$ is the contribution from a ``single" genus $0$ Gopakumar-Vafa invariant (i.e. a contribution of $1$ to $n_{0, \beta}$) to Gromov-Witten invariants of genus $0$. This also holds for the term $-1/6$ up to a factor of $2$; in this case the contribution is to Gromov-Witten invariants of genus $1$.  
\end{rmk}
\begin{rmk} In order to help comparison of this result with the computation by Bridgeland and Iwaki in \cite{bridRH} Section 6.3, we note that setting $s = 1$ simply gets rid of the formal parameters needed to make $X(t)$ well-defined in the sense of formal power series, while the specialisation $x_{(m, \beta, n)} = 1$ corresponds to the choice of asymptotic behaviour $\xi(\gamma) = 1$ for $\Omega(\gamma) \neq 0$ appearing in \cite{bridRH}, Theorem 3.2. Finally, the differential equation satisfied by $\operatorname{F}_{\beta}(u)$ is very close to the notion of tau function discussed in \cite{bridRH} Section 3.4, although there is no analogue of our auxiliary variable $\eps$ in that case. 
\end{rmk}
\begin{proof}[Proof of Theorem \ref{MainThm}] According to Definition \ref{FrequDef} we have
\begin{align*}
&\big(\operatorname{F}^+_{m, \beta, n}(t, Z) + \operatorname{F}^-_{m, \beta, n}(t, Z)\big) | \big( s = x_{(m, \beta, n)} = 1\big)=\\
&= \Omega(m,\beta,n)\frac{t}{2\pi i}  \int_{\R_{>0}Z(m,\beta,n)} \frac{\log(1 - e^{- Z(m,\beta,n)/z})}{z-t}\frac{dz}{z}\\
& -\Omega(m,\beta,n)\frac{t}{2\pi i}  \int_{\R_{<0}Z(m,\beta,n)} \frac{\log(1 - e^{Z(m,\beta,n)/z})}{z-t}\frac{dz}{z}.
\end{align*}
Assuming the condition \textbf{C$\chi$} we have $\Omega(m,\beta,n) = \Omega_n(\beta)$ for all $m \in \Z$. Moreover we recall that the central charge is given explicitly by
\begin{align*}
Z(m, \beta, n) = \int_{\beta} (i\omega) - m + n\eps = v_{\beta} - m + n\eps.
\end{align*} 
Therefore
\begin{align*}
&\big(\operatorname{F}^+_{m, \beta, n}(t,Z) + \operatorname{F}^-_{m, \beta, n}(t,Z)\big) | \big( s = x_{(m, \beta, n)} = 1\big)=\\
&= \Omega_n(\beta)\frac{t}{2\pi i}  \int_{\R_{>0}(v_{\beta} - m + n\eps)} \frac{\log(1 - e^{- (v_{\beta} - m + n\eps)/z)})}{z-t}\frac{dz}{z}\\
& -\Omega_n(\beta)\frac{t}{2\pi i}  \int_{\R_{<0}(v_{\beta} - m + n\eps)} \frac{\log(1 - e^{(v_{\beta} - m + n\eps)/z)})}{z-t}\frac{dz}{z}.
\end{align*}
Recall we are assuming that $t$ is a real variable and $\eps$ is a real parameter. Therefore it is possible to deform the integration ray $\R_{>0}(v_{\beta} - m + n\eps)$ to the ray $i\R_{>0}$ within the open upper half-plane, without enclosing poles of the integrand. Similarly we can deform $\R_{<0}(v_{\beta} - m + n\eps)$ to $i\R_{<0}v_{\beta}$. As a consequence we may rewrite
\begin{align*}
&\big(\operatorname{F}^+_{m, \beta, n}(t,Z) + \operatorname{F}^-_{m, \beta, n}(t,Z)\big) | \big( s = x_{(m, \beta, n)} = 1\big)=\\
&= \Omega_n(\beta)\frac{t}{2\pi i}  \int_{i\R_{>0} } \frac{\log(1 - e^{- (v_{\beta} - m + n\eps)/z)})}{z-t}\frac{dz}{z}\\
& -\Omega_n(\beta)\frac{t}{2\pi i}  \int_{i\R_{<0} } \frac{\log(1 - e^{(v_{\beta} - m + n\eps)/z)})}{z-t}\frac{dz}{z}.
\end{align*}
Expanding the logarithm into a convergent power series and summing over the finitely many values of $n \in \Z$ for which $\Omega_n(\beta)\neq 0$, as well as over $|m| \leq M$ for a cutoff integer $M >0$, gives a well defined function
\begin{align*} 
&\operatorname{F}^{\leq M}_{\beta}(t, Z) = \sum_{n \in \Z, |m| \leq M} \big(\operatorname{F}^+_{m, \beta, n}(t,Z) + \operatorname{F}^-_{m, \beta, n}(t,Z)\big) | \big( s = x_{(m, \beta, n)} = 1\big)\\
&=-\sum_{n \in \Z} \frac{1}{2\pi i} \Omega_{n}(\beta) \int_{i\R_{>0} } \frac{dz}{z}\frac{t}{z-t}\sum_{k>0}\frac{1}{k} e^{ k n \frac{\eps}{z}}   \sum_{|m|\leq M} e^{-k (v_{\beta}-m)/z} \\
& +\sum_{n \in \Z} \frac{1}{2\pi i} \Omega_{n}(\beta) \int_{i\R_{<0} } \frac{dz}{z}\frac{t}{z-t}\sum_{k>0}\frac{1}{k}e^{ k n \frac{\eps}{z}}   \sum_{|m|\leq M}e^{k(v_{\beta}-m)/z}.
\end{align*}
Making the change of variable $z \mapsto -z$ in the integrals over $i\R_{<0}$ gives
\begin{align*} 
\operatorname{F}^{\leq M}_{\beta}(t, Z) &=-\sum_{n \in \Z} \frac{1}{2\pi i} \Omega_{n}(\beta) \int_{i\R_{>0}} \frac{dz}{z}\frac{t}{z-t}\sum_{k>0}\frac{1}{k}e^{ k  n \frac{\eps}{z}}   \sum_{|m| \leq M} e^{-k (v_{\beta}-m)/z} \\
&  +\sum_{n \in \Z} \frac{1}{2\pi i} \Omega_{n}(\beta) \int_{i\R_{>0}} \frac{dz}{z}\frac{t}{-z-t}\sum_{k>0}\frac{1}{k}e^{- k n \frac{\eps}{z}}   \sum_{|m|\leq M}e^{-k(v_{\beta}-m)/z}.
\end{align*}
By Lemma \ref{nSymmLem}, i.e. by the symmetry $\Omega_{n}(\beta) = \Omega_{-n}(\beta)$, we have
\begin{align*} 
\operatorname{F}^{\leq M}_{\beta}(t, Z) &=-\sum_{n \in \Z} \frac{1}{2\pi i} \Omega_{n}(\beta) \int_{i\R_{>0} } \frac{dz}{z}\frac{t}{z-t}\sum_{k>0}\frac{1}{k}e^{ k  n \frac{\eps}{z}}   \sum_{|m|\leq M} e^{-k (v_{\beta}-m)/z} \\
&  +\sum_{n \in \Z} \frac{1}{2\pi i} \Omega_{n}(\beta) \int_{i\R_{>0} } \frac{dz}{z}\frac{t}{-z-t}\sum_{k>0}\frac{1}{k}e^{k n \frac{\eps}{z}}   \sum_{|m|\leq M}e^{-k(v_{\beta}-m)/z} 
\end{align*}
(notice the change of sign in the second integrand). Collecting terms we find 
\begin{align*} 
&\operatorname{F}^{\leq M}_{\beta}(t, Z)=\\
&=-\sum_{n \in \Z} \frac{1}{2\pi i} \Omega_{n}(\beta) \sum_{k>0}\frac{1}{k}\int_{i\R_{>0}} \frac{dz}{z}\left( \frac{t}{z-t} - \frac{t}{-z-t}\right) e^{  k n \frac{\eps}{z}}\sum_{|m|\leq M} e^{-k (v_{\beta}-m)/z}\\
&=-\sum_{n \in \Z} \frac{1}{2\pi i} \Omega_{n}(\beta) \sum_{k>0}\frac{1}{k}\int_{i\R_{>0}} dz \frac{2t}{z^2-t^2} e^{  k n \frac{\eps}{z}}\sum_{|m|\leq M} e^{-k (v_{\beta}-m)/z}.
\end{align*}
Writing $z = i p$ for a real variable $p$ gives
\begin{align*} 
\operatorname{F}^{\leq M}_{\beta}(t, Z)
&= \frac{1}{\pi} \sum_{n \in \Z} \Omega_{n}(\beta)\sum_{k>0}\frac{1}{k} \int_{\R_{>0}} d p \frac{ t}{p^2 + t^2} e^{-i k n \frac{\eps}{p}}\sum_{|m|\leq M} e^{i k (v_{\beta}-m)/p}\\
&= \frac{1}{\pi} \sum_{n \in \Z}  \Omega_{n}(\beta)\sum_{k>0}\frac{1}{k}\int_{\R_{>0}} dp \frac{ t}{1+ (p t)^2} e^{-i k n p \eps}\sum_{|m|\leq M} e^{i k (v_{\beta}-m) p},
\end{align*}
where the last equality follows from the change of variable $p \mapsto p^{-1}$. As $M \to +\infty$ the function
\begin{equation*}
\frac{k}{2\pi}\sum_{|m| \leq M} e^{ i p k m}  
\end{equation*}
converges in the sense of tempered distributions to the Dirac comb 
\begin{equation*}
\sum_{m \in \Z} \delta(p - \frac{2\pi}{k} m).
\end{equation*}
Moreover for fixed $t$ the function 
\begin{equation*}
\left| \frac{ t}{1+ (p t)^2} e^{-i k n p \eps}e^{i k v_{\beta}  p}\right| 
\end{equation*}
is of rapid decay as $p \to +\infty$. So we find a well defined limit function
\begin{align}\label{SumOver_n_Equ}
\nonumber\operatorname{F}_{\beta}(t, Z) &= \lim_{M \to +\infty}\operatorname{F}^{\leq M}_{\beta}(t, Z) \\&= 2\sum_{n \in \Z}   \Omega_{n}(\beta)\sum_{k>0}\frac{1}{k^2}  \sum_{m > 0} \frac{ t}{1+ (\frac{2 \pi m}{k} t)^2}  e^{ - 2\pi i m n \eps } e^{2\pi i m  v_{\beta} }. 
\end{align}
Note that we are only summing over $m > 0$ since each integral over $p$ is performed along $\R_{>0}$, so contributions from  $\delta(p - \frac{2\pi}{k} m)$ with $m \leq 0$ vanish. 

Now the crucial property \eqref{SheafGV}, expressing the relation between sheaf-theoretic and Gopakumar-Vafa invariants, and the fact that by Definition \ref{MauTodaBPSDef} we have  $\Omega_n(\beta) = (-1)^n\mathcal{S}_{n}(\beta)$, imply
\begin{align}\label{mautodaGVSigned}
\nonumber\sum_{n\in\Z} \Omega_{n}(\beta)y^n &= \sum_{g\geq0} n_{g, \beta} ((-y)^{1/2} + (-y)^{-1/2})^{2g}\\
\nonumber &= \sum_{g\geq0} n_{g, \beta} (i y ^{1/2} - i y^{-1/2})^{2g}\\
&= \sum_{g\geq0} n_{g, \beta} (-1)^{g}(y ^{1/2} - y^{-1/2})^{2g}.
\end{align}
Note that in fact this identity is independent of the choice $(-1)^{1/2} = \pm i$. According to \eqref{mautodaGVSigned}, for fixed $k, m$ we can perform the sum over $n$ in \eqref{SumOver_n_Equ} as
\begin{align*}
&\sum_{n\in \Z} \Omega_{n}(\beta) e^{- 2\pi i m n \eps } =  \sum_{g \geq 0} n_{g,\beta} (-1)^g(e^{ -\pi i m \eps } - e^{\pi i m \eps})^{2g}.
\end{align*}
This shows
\begin{equation}\label{SumOver_g_Equ}
\operatorname{F}_{\beta}(t, Z)= 2 \sum_{g\geq0}n_{g,\beta} (-1)^g \sum_{m>0}(e^{ - i\pi m \eps } - e^{i \pi m \eps})^{2g} e^{2\pi i m  v_{\beta} } \sum_{k>0}\frac{1}{k^2} \frac{ t}{1+ (\frac{2 \pi m}{k} t)^2}. 
\end{equation}
Expanding the rational function of $t$ and using the well-known identity 
\begin{equation*}
\sum_{k>0}\frac{1}{k^{2s}} = \zeta(2s) =  \frac{|B_{2s}|(2\pi)^{2s}}{2(2s)!} 
\end{equation*}
we obtain 
\begin{align*} 
\operatorname{F}_{\beta}(t, Z) &= \frac{1}{\pi} \sum_{g\geq0}n_{g,\beta} (-1)^g\sum_{m>0}(e^{ -i\pi m \eps } - e^{i\pi m \eps})^{2g} e^{2\pi i m  v_{\beta} }\\&\quad\quad\quad \sum_{r\geq0} (-1)^r m^{2r} (2\pi t)^{2r+1} \sum_{k>0} \frac{1}{k^{2r+2}}\\
&= \frac{1}{\pi} \sum_{g\geq0}n_{g,\beta} (-1)^g\sum_{m>0}(e^{ - i\pi  m \eps } - e^{i\pi m \eps})^{2g} e^{2\pi i m  v_{\beta} }\\&\quad\quad\quad \sum_{s\geq1} (-1)^{s-1} m^{2s-2} (2\pi t)^{2s-1} \sum_{k>0} \frac{1}{k^{2s}}\\
&= \frac{1}{\pi} \sum_{g\geq0}n_{g,\beta} (-1)^g\sum_{m>0}(e^{ - i\pi  m \eps } - e^{i\pi m \eps})^{2g} e^{2\pi i m  v_{\beta} }\\&\quad\quad\quad  \sum_{s\geq1} (-1)^{s-1} m^{2s-2}  \frac{|B_{2s}|(2\pi)^{2s}}{2(2s)!}(2\pi t)^{2s-1}.
\end{align*} 
In particular we note the identity
\begin{align}\label{DiffEquProof}
\nonumber\int\frac{\del}{\del t}\operatorname{F}_{\beta}(t, Z) dv_{\beta} &= -2\pi i \sum_{g\geq0}n_{g,\beta} (-1)^g\sum_{m > 0}\frac{1}{m}(e^{ -i\pi m \eps } - e^{ i\pi m \eps})^{2g} e^{2\pi i m  v_{\beta} }\\&  \quad\quad\sum_{s\geq1} (-1)^{s-1}  \frac{|B_{2s}|}{(2s)(2s-2)!}(4\pi^2 m t)^{2s-2}.
\end{align}
Let us introduce a new formal parameter $u$, which plays the role of the Gromov-Witten coupling. It is related to the variables $\eps$, $t$ by
\begin{equation*}
\eps = \frac{u}{2\pi},\, t = \frac{u}{ (2\pi)^2}.
\end{equation*} 
With this specialisation, we may rewrite \eqref{DiffEquProof} as the differential equation 
\begin{align*}
&\frac{\del}{\del t}\operatorname{F}_{\beta}(t, Z)|\big(\eps = \frac{u}{2\pi}, t = \frac{u}{ (2\pi)^2}\big) = \\&=-2\pi i\frac{\del}{\del v_{\beta}}\sum_{g\geq0}n_{g,\beta} (-1)^g\sum_{m>0}\frac{1}{m}(2 i \sin(mu/2))^{2g} \\&\quad\quad\quad\,\,\,\sum_{s\geq1} (-1)^{s-1}  \frac{|B_{2s}|}{(2s)(2s-2)!}(m u)^{2s-2}  e^{2\pi i m  v_{\beta} }.
\end{align*}
We use the standard expansion around $x = 0$ 
\begin{equation*}
(2\sin(x/2))^{-2} = \frac{1}{x^2} - \frac{1}{12} + \sum_{s\geq2} (-1)^{s-1}  \frac{|B_{2s}|}{(2s)(2s-2)!}x^{2s-2} 
\end{equation*} 
and the identity $ |B_2| = \frac{1}{6} $ to compute
\begin{equation*}
\sum_{s\geq1} (-1)^{s-1}  \frac{|B_{2s}|}{(2s)(2s-2)!}(m u)^{2s-2} = (2\sin(mu/2))^{-2} - \frac{1}{m^2 u^2} + \frac{1}{6}.
\end{equation*}
Thus the differential equation becomes, after a little simplification
\begin{align*}
&\frac{\del}{\del t}\operatorname{F}_{\beta}(t, Z)|\big(\eps = \frac{u}{2\pi}, t = \frac{u}{ (2\pi)^2}\big) = \\&=-2\pi i \frac{\del}{\del v_{\beta}}\sum_{g\geq0}n_{g,\beta} (-1)^g\sum_{m>0}\frac{1}{m}(2 i \sin(mu/2))^{2g} \\&\quad\quad\quad\left((2\sin(mu/2))^{-2} - \frac{1}{m^2 u^2} + \frac{1}{6}\right)  e^{2\pi i m  v_{\beta} }.
\end{align*}
This completes the proof.
\end{proof}

\begin{rmk} Fixing $t = u/(2\pi)^2$ and taking the limit $\eps \to 0$ in \eqref{DiffEquProof} gives just the genus $0$ contribution
 \begin{equation*} 
2\pi i n_{0,\beta} \sum_{m>0}\frac{1}{m}\left((2\sin(mu/2))^{-2} - \frac{1}{m^2 u^2} + \frac{1}{6}\right) e^{2\pi i m  v_{\beta} }.
\end{equation*}
This particular limit essentially reproduces Bridgeland and Iwaki's calculation in \cite{bridRH} Section 6.3. 
\end{rmk}
\section{Large $t$ behaviour}\label{Large_t_Sec}
Let us recall the identity \eqref{SumOver_g_Equ} from the proof of Theorem \ref{MainThm}:
\begin{equation*} 
\operatorname{F}_{\beta}(t, Z)= 2 \sum_{g\geq0}n_{g,\beta} (-1)^g \sum_{m>0}(e^{ - i\pi m \eps } - e^{i \pi m \eps})^{2g} e^{2\pi i m  v_{\beta} } \sum_{k>0}\frac{1}{k^2} \frac{ t}{1+ (\frac{2 \pi m}{k} t)^2}. 
\end{equation*}
Since the sum over the genus $g$ is finite, this shows that we may also regard the function $\operatorname{F}_{\beta}(t, Z)$ as a well-defined function of $t \in \R_{>0}$ with values in $\C[e^{\pm i \pi \eps}]\pow{e^{2\pi i v_{\beta}}}$. 

Therefore it makes sense to consider the behaviour of $\operatorname{F}_{\beta}(t, Z)$ as $t \to +\infty$, as a function with values in $\C[e^{\pm i \pi \eps}]\pow{e^{2\pi i v_{\beta}}}$. Here we compute the leading order term for this expansion. As we observed in the Introduction, this is somewhat reminiscent of computations of the entries of the ``central connection matrix" for the Dubrovin connection in quantum cohomology \cite{cotti, dubrovin, gamma}.

So rather than expanding the rational function of $t$ as a formal power series around $t = 0$ as in the proof of Theorem \ref{MainThm}, we consider the function of $t$ defined by the series
\begin{equation*}
\sum_{k>0}\frac{1}{k^2} \frac{ t}{1+ (\frac{2 \pi m}{k} t)^2} = \sum_{k>0} \frac{ t}{k^2+  (2 \pi m t)^2}.
\end{equation*} 
Notice the identity
\begin{equation*}
\sum_{k > 0} \frac{1}{k^2 + x^2} = \frac{\pi  x \coth (\pi  x)-1}{2 x^2}.
\end{equation*}
In particular we have
\begin{equation*}
\sum_{k>0}\frac{ t}{k^2+  (2 \pi m t)^2} = \frac{2 \pi ^2 m t \coth \left(2 \pi ^2 m t\right)-1}{8 \pi ^2 m^2 t}
\end{equation*}
Taking the limit for real $t\to +\infty$ shows that we have
\begin{equation*}
\sum_{k>0}\frac{ t}{k^2+  (2 \pi m t)^2} = \frac{1}{4 m} + O(t^{-1}),
\end{equation*}
uniformly in $m > 0$. In particular we find as $t\to +\infty$
\begin{equation*} 
\operatorname{F}_{\beta}(t, Z)= \frac{1}{2} \sum_{g\geq0}n_{g,\beta} (-1)^g \sum_{m>0}\frac{1}{m}(e^{ - i\pi m \eps } - e^{i \pi m \eps})^{2g} e^{2\pi i m  v_{\beta} } + O(t^{-1}).
\end{equation*}
Expanding
\begin{align*}
(e^{ - i\pi m \eps } - e^{i \pi m \eps})^{2g} = \sum^{2g}_{h = 0} \binom{2g}{h} (-1)^h e^{ 2 \pi i(g-h) m \eps}  
\end{align*}
we find 
\begin{align*}
&\sum_{m>0}\frac{1}{m}(e^{ - i\pi m \eps } - e^{i \pi m \eps})^{2g} e^{2\pi i m  v_{\beta} } \\&= \sum^{2g}_{h = 0} \binom{2g}{h} (-1)^h \sum_{m>0}\frac{1}{m}e^{ 2 \pi i m ( (g-h)\eps + v_{\beta})}\\
&= -\sum^{2g}_{h = 0} \binom{2g}{h} (-1)^h \log(1 - e^{ 2 \pi i ( (g-h)\eps + v_{\beta})} ).
\end{align*}
Thus we have shown that there is an expansion as $t \to +\infty$  
\begin{align*}
\exp(\bra [D], \beta \ket\operatorname{F}_{\beta}(t, Z)) &= \prod_{g \geq 0} \prod^{2g}_{h = 0}(1 - e^{ 2 \pi i ( (g-h)\eps + v_{\beta})})^{-\frac{1}{2} (-1)^{g+h} \binom{2g}{h} n_{g,\beta} \bra [D], \beta \ket}\\ &+ O(t^{-1}),
\end{align*}
where the product over $g$ is finite. We observe that the leading order term can be thought of as a formal family of automorphisms of an algebraic torus, of the type studied in \cite{gps}. In other words it seems natural to associate to this leading order term an automorphism of the ring $\C[x_{[D]}, e^{\pm 2\pi i\eps }]\pow{e^{2\pi i v_{\beta}}}$, where $[D]$ varies along divisor classes in $H_4(X, \Z)$, which is $\C\pow{e^{2\pi i v_{\beta}}}$-linear and only acts nontrivially on the generators $x_{[D]}$, with the rule 
\begin{equation*}
x_{[D]} \mapsto x_{[D]} \prod_{g \geq 0} \prod^{2g}_{h = 0}(1 - e^{2\pi i v_{\beta}} e^{ 2 \pi i (g-h)\eps})^{-\frac{1}{2} (-1)^{g+h} \binom{2g}{h} n_{g,\beta} \bra [D], \beta \ket}.
\end{equation*}

\end{document}